\documentclass[11pt]{amsart}
\usepackage{amsmath,amssymb,amsfonts,amsthm,latexsym,soul,cite,mathrsfs}
\usepackage{mathabx}
\usepackage{enumitem}

\usepackage{color,enumitem,graphicx}
\usepackage[colorlinks=true,urlcolor=blue,
citecolor=red,linkcolor=blue,linktocpage,pdfpagelabels,
bookmarksnumbered,bookmarksopen]{hyperref}
\usepackage[english]{babel}

\usepackage[left=2.9cm,right=2.9cm,top=2.8cm,bottom=2.8cm]{geometry}
\usepackage[hyperpageref]{backref}

\usepackage[colorinlistoftodos]{todonotes}
\makeatletter
\providecommand\@dotsep{5}
\def\listtodoname{List of Todos}
\def\listoftodos{\@starttoc{tdo}\listtodoname}
\makeatother

\numberwithin{equation}{section}

\newtheorem{theorem}{Theorem}[section]

\newtheorem{lemma}[theorem]{Lemma}
\newtheorem{corollary}[theorem]{Corollary}
\newtheorem{example}[theorem]{Example}

\newtheorem{remark}[theorem]{Remark}

\begin{document}

	\title [Norms and Non-Equivalence in Infinite-Dimensional Banach Spaces]{Norms and Non-Equivalence in Infinite-Dimensional Banach Spaces}
	
	\author{Renan J. S. Isneri}
	\author{Josias V. Baca}
	\author{Lucas M. Fernandes}
	
	\address[Renan J. S. Isneri]
	{\newline\indent Unidade Acad\^emica de Matem\'atica
		\newline\indent 
		Universidade Federal de Campina Grande,
		\newline\indent
		58429-970, Campina Grande - PB - Brazil}
	\email{\href{renanisneri@mat.ufcg.edu.br}{renanisneri@mat.ufcg.edu.br}}
	
	\address[Josias V. Baca]
	{\newline\indent Unidade Acad\^emica de Matem\'atica
		\newline\indent 
		Universidade Federal de Campina Grande,
		\newline\indent
		58429-970, Campina Grande - PB - Brazil}
	\email{\href{jverabaca@gmail.com}{jverabaca@gmail.com}}
	
	\address[Lucas M. Fernandes]
	{\newline\indent Departamento de Matem\'atica
		\newline\indent 
		Universidade Regional do Cariri,
		\newline\indent
		63150-000, Campos Sales - CE - Brazil}
	\email{\href{lucas.mfernandes@urca.br}{lucas.mfernandes@urca.br}}
	
\pretolerance10000
	
\begin{abstract}
	\noindent This work explores the interaction between different norms in infinite-dimensional vector spaces, focusing on their impact on Banach space structures and topological properties. We examine norms induced by bijective linear maps, the existence of non-equivalent norms in Banach spaces, and the role of Hamel bases in normed spaces. 
\end{abstract}
	
	
\maketitle
	
\section{Introduction}

One of the first questions in Functional Analysis is: {\it Given a real or complex vector space, can we define a norm on it?} The answer is yes, and the construction is based on the existence of a Hamel basis. In fact, if $E$ is a nontrivial real or complex vector space, Zorn's Lemma ensures that there is a Hamel basis $\mathcal{B}$ for $E$ (see \cite[Theorem 2.42]{Knapp}). In true, by Proposition 2.1 of \cite{Bartoszynski}, every Banach space $E$ over the real or complex field possesses $2^{|E|}$ distinct Hamel bases, where $|E|$ denotes the cardinality of $E$. This means that any vector $u \in E$ can be uniquely written as a finite sum  
$$
u =\sum_{k=1}^{n} \alpha_k e_k= \alpha_{1} e_{1} + \cdots + \alpha_{n} e_{n},
$$  
where $\alpha_1, \dots, \alpha_n$ are scalars, and $e_1, \dots, e_n$ are vectors in the basis $\mathcal{B}$. Using this representation, we can define norms on $E$ as follows. For $1 \leq p < \infty$, we set:  
\begin{equation}\label{p-norm}
	\| u \|_{p} = \left( \sum_{k=1}^{n} |\alpha_k|^p  \right)^{1/p}, \quad  u \in E.
\end{equation}
For $p = \infty$, we use: 
\begin{equation}\label{infty-norm}
	\| u \|_{\infty} = \max_{ 1 \leq k \leq n} |\alpha_k|, \quad  u \in E.
\end{equation} 
From now on, we will refer to the norm $\| \cdot \|_p$ on $E$ as the {\it $p$-norm}. In particular, when $p = 2$, the norm $\| \cdot \|_2$ defines an inner product on $E$. Although this shows that any vector space can be given a norm, it does not necessarily mean that $E$ becomes a Banach space. This leads to a natural question: {\it What conditions must $E$ satisfy for $( E, \| \cdot \|_p )$ to be a Banach space?} If $E$ has finite dimension, then it is always a Banach space, since any finite-dimensional normed space is Banach \cite[Theorem 2.4-2]{Kreyzig}. However, in infinite dimension, the answer is not so simple. One of the goals of this work is to prove the following result:  

\begin{theorem}\label{theorem1.1}  
	Let $E$ be a real or complex vector space and $1 \leq p \leq \infty$. Then $( E , \| \cdot \|_p )$ is a Banach space if and only if $E$ has finite dimension.
\end{theorem}

With this result, we can conclude that in every infinite-dimensional vector space, there is at least one norm that does not make it a Banach space. This observation naturally leads to another question: {\it Does every vector space have a norm that makes it a Banach space?} The answer is no. One key property of a Banach space, which follows from Baire's Category Theorem, is that its Hamel basis must be either finite or uncountable \cite[Corollary 5.23]{Aliprantis}. Because of this, for example, the space of polynomials $\mathcal{P}(\mathbb{R})$ cannot be a Banach space under any norm, since its Hamel basis is the countable set $\{1, t, t^2, \dots\}$. More generally, what we do know from the Completeness Theorem \cite{Kreyzig} is that any normed space can be isometrically identified with a dense subspace of some Banach space. This limitation motivates the search for tools to define a norm on a vector space $E$ with an uncountable Hamel basis, ensuring that $E$ becomes a Banach space under that norm. One way to approach this problem is to find a Banach space $F$ that is isomorphic to $E$. More precisely:

\begin{theorem}\label{theorem1.2}  
	Let $E$ be a real or complex vector space. If there exists a Banach space $F$ and a bijective linear operator $T: E \to F$, then the function $\| u \| = \| Tu \|_F$ defines a norm on $E$ that makes it a Banach space.
\end{theorem}

Determining the existence of a Banach space $F$ that satisfies the hypotheses of Theorem \ref{theorem1.2} is far from trivial. One possible approach to addressing this question is to show that for every cardinal number $\kappa \geq \mathfrak{c}$, where $\mathfrak{c}$ denotes the cardinality of the continuum, there exists a Banach space $F$ with cardinality $\kappa$. In such a case, any Hamel basis of $F$ would necessarily have cardinality $\kappa$, as established in Theorem 3.5 of \cite{Lorenz}. This would imply, in particular, that for any vector space $E$ with a Hamel basis with cardinality $\kappa$, it would be possible to construct a linear isomorphism between $E$ and $F$, thus endowing $E$ with a Banach space structure. However, this approach fails when $ \kappa = \aleph_0 $, the cardinality of $ \mathbb{N} $, as previously observed. This leads to an intriguing question: {\it Does there exist some cardinal $ \kappa>\aleph_0 $ for which no vector space with a Hamel basis of that cardinality can be endowed with a norm that turns it into a Banach space?} To the best of the authors' knowledge, no explicit example of such a cardinal $ \kappa $ is currently known. This question was originally posed by Laugwitz in 1955 \cite{Laugwitz} and was partially answered in \cite{Arent}, where Theorem 4.7 provides a result under the assumption of the Generalized Continuum Hypothesis.  

The theory of norms in vector spaces raises several interesting questions, especially in more abstract contexts such as Functional Analysis and the Theory of Transfinite Numbers. A central question in this study is the relationship between different norms defined on the same vector space. Suppose that a vector space $E$ admits two norms, denoted by $\| \cdot \|_1$ and $\| \cdot \|_2$. A natural question arises: {\it Are these norms equivalent?} In Functional Analysis, we say that the norms $\| \cdot \|_1$ and $\| \cdot \|_2$ are equivalent, and we write $\| \cdot \|_1 \sim \| \cdot \|_2$, if there exist positive constants $c_1$ and $c_2$ such that 
$$
c_1 \| u \|_1 \leq \| u \|_2 \leq c_2 \| u \|_1, \quad \text{for all } u \in E.
$$  
If this relation is not satisfied, we say that the norms are not equivalent, which is represented symbolically as $\| \cdot \|_1 \not\sim \| \cdot \|_2$. The equivalence of norms has a fundamental property: if two norms are equivalent, they induce the same topology on the vector space. This means that although they may measure vectors in different ways, they preserve essential topological properties such as convergence, compactness, and separability.

In the study of Functional Analysis, a fundamental property of finite-dimensional spaces is that all norms are equivalent \cite[Theorem 2.4-5]{Kreyzig}. However, when we move to the study of infinite-dimensional spaces, the situation changes dramatically, and norm equivalence is no longer universal. A classic example illustrating this difference is the space of continuous functions $C([0,1], \mathbb{R})$. In this space, two widely used norms are the uniform norm $\|\cdot\|_{\infty}$ and the integral norm $\|\cdot\|_1$, defined as:  
$$
\| f \|_{\infty} = \max_{t \in [0,1]} |f(t)|\quad\text{and}\quad  
\| f \|_{1} = \int_{0}^{1} |f(s)| \, \mathrm{d}s.
$$  
These norms capture different aspects of $f$, and they are not equivalent. More precisely, while every continuous function $f$ in $C([0,1], \mathbb{R})$ satisfies the inequality $\| f \|_{1} \leq \| f \|_{\infty}$, the converse does not hold: there is no constant $c > 0$ such that $\| f \|_{\infty} \leq c \| f \|_{1}$ for all continuous functions $f$. This distinction leads to different notions of convergence and compactness. For instance, a set that is compact with respect to one norm may not be compact with respect to the other, which has direct implications for the study of continuous linear operators and the analysis of sequences or series in $E$. This difference affects the topological structure of sets, the nature of continuous functions and compact operators, and the interpretation of key concepts such as completeness, separability, and duality.

A remarkable result in the theory of infinite-dimensional spaces concerns the relationship between equivalent norms and the structure of Banach spaces, based on the Open Mapping Theorem (See \cite[Corollary 2.8]{Brezis}). Specifically, consider a normed vector space $E$ equipped with two distinct norms, $\| \cdot \|_1$ and $\| \cdot \|_2$, such that both $(E, \| \cdot \|_1)$ and $(E, \| \cdot \|_2)$ are Banach spaces. The result states that if there exists a positive constant $c > 0$ such that  
$$
\| u \|_1 \leq c \| u \|_2  \quad \text{for all} \;\; u \in E,
$$  
then the norms $\| \cdot \|_1$ and $\|\cdot \|_2$ are equivalent. However, this theorem does not cover all cases, and its applicability crucially depends on the completeness of the underlying normed spaces. A classic counterexample is the space $C([0,1], \mathbb{R})$, equipped with the norm $\| \cdot \|_{\infty}$ and the integral norm $\| \cdot \|_1$. While $(C([0,1], \mathbb{R}), \| \cdot \|_{\infty})$ forms a Banach space, the same is not true for $(C([0,1], \mathbb{R}), \|\cdot\|_1)$.  

Motivated by the previous discussions, an intriguing question arises: {\it Is it possible to find a vector space equipped with two norms, where both define Banach structures, but the norms are not equivalent?} The answer to this question reveals one of the most fascinating aspects of functional analysis in infinite-dimensional spaces, and in fact, it shows that every infinite-dimensional Banach space satisfies this property. More specifically, the following theorem can be proven:

\begin{theorem}\label{theorem1.3}  
	For every infinite-dimensional real or complex Banach space $(E, \| \cdot \|_1)$, there exists another norm $\| \cdot \|_2$ such that $(E, \| \cdot \|_2)$ is also a Banach space, but the norms $\| \cdot \|_1$ and $\| \cdot \|_2$ are not equivalent.
\end{theorem} 

This result is well known in the literature and is mentioned, for example, in \cite{Arent}, where the idea of the proof is outlined. However, a detailed proof is not presented. Given the fundamental nature of this theorem  and the richness of the argument, we believe that providing a complete proof adapted to our context, together with a discussion of its consequences, adds value to our work.

Theorem \ref{theorem1.3} reveals the richness and complexity of Banach spaces in infinite dimension. Since non-equivalent norms can coexist in the same vector space, a natural question arises: {\it To what extent do these norms preserve (or fail to preserve) the same topological properties?} In general, in infinite-dimensional spaces, topological properties are often not preserved under non-equivalent norms. For example, consider $(E, \|\cdot\|)$ a separable Banach space. By Corollary 6.8 in \cite{Carothers}, there exists an isometric isomorphism $T: (E, \|\cdot\|) \to F$, where $F$ is a subspace of $(\ell^{\infty}, \|\cdot \|_{\infty})$, where $\| \cdot\|_{\infty}$ denotes the usual norm in $\ell^{\infty}$. However, for any other norm $\|\cdot\|_*$ on $E$ such that $(E, \|\cdot\|_*)$ is a Banach space and $\|\cdot\|_* \not\sim \|\cdot\|$, the linear operator $\tilde{T}: (E, \|\cdot\|_*) \to F$ defined by $\tilde{T}(u)=T(u)$ is not continuous. This happens because there is no constant $c > 0$ that satisfies the inequality 
$$
\|\tilde{T}(u)\|_{\infty} = \|u\| \leq c\|u\|_*,\quad \text{for all }u\in E.
$$
From this perspective, we will prove the following:

\begin{theorem}\label{theorem1.4} 
	For every real or complex separable Banach space $(E, \|\cdot\|_1)$ of infinite dimension, there exists another norm $\|\cdot\|_2$ such that $(E, \|\cdot\|_2)$ is a non-separable Banach space, and $\|\cdot\|_1 \not\sim \|\cdot\|_2$.
\end{theorem}

This result highlights that the separability of a vector space under one norm may be lost under another, even if both norms define the space of Banach. This fact also highlights the importance of carefully selecting the most appropriate norm for each context. Such a choice is particularly relevant in applications such as the study of Calculus of Variations and Differential Equations, where topological properties play a crucial role in problem analysis and solution. 

Note that separability is a topological property preserved under equivalent norms. In other words, if a normed space $ E $ satisfies the property $ P $ with respect to the norm $ \|\cdot\|_1 $, then $ E $ will also satisfy $ P $ with respect to any equivalent norm $ \|\cdot\|_2 $. Other examples of properties that are invariant under equivalent norms include completeness, reflexivity, the approximation property, the existence of a Schauder basis, and the Schur property. Based on this principle, we can state a more general result than Theorem \ref{theorem1.4}, considering topological properties that remain unchanged under equivalent norms:  

\begin{theorem}\label{NewTheorem} 
	Let $(E, \|\cdot\|_1)$ be a real or complex Banach space satisfying a property $ P $ that is preserved under equivalent norms and under isometries. If there exists a banach space $ F $ with the same cardinality as $ E $ that does not satisfy property $ P $, then there exists a norm $\|\cdot\|_2$ on $ E $ that is not equivalent to $\|\cdot\|_1$ and such that $(E, \|\cdot\|_2)$ does not satisfy property $ P $.  
\end{theorem} 

Here, the property $P$ is preserved under isometries, meaning that if a space has property $P$, then any isometric space also has the property $P$. We would like to highlight that not every topological property is preserved under equivalent norms, an example of this is uniform convexity. However, this issue has been explored in different contexts, and a relevant study \cite{Kikianty} introduces a weaker notion than that of equivalent norms, called angularly equivalent norms. In this new setting, the property of uniformly convex spaces remains invariant. For more details, see \cite{Kikianty2,Kikianty}. 

The primary goal of this work is to address the questions raised in this introduction, as well as to explore other intriguing problems that naturally arise in this context, using fundamental concepts of Functional Analysis. Although some of the topics discussed here are classical, the integrated and cohesive approach presented is, to the best of our knowledge, novel. We would like to emphasize that this work provides a valuable resource for students, researchers, and educators interested in Functional Analysis. 

This work is structured as follows: In Section \ref{Sec2}, we discuss fundamental questions about norms in vector spaces induced by a bijective linear map, exploring their properties and implications. As part of this discussion, we establish Theorem \ref{theorem1.2}. Section \ref{Sec3} addresses the existence of non-equivalent norms in infinite-dimensional Banach spaces, examining how these norms affect key topological properties. In this section, we prove Theorems \ref{theorem1.3} and \ref{theorem1.4}. Finally, in Section \ref{Sec4}, we investigate whether norms generated by a Hamel basis in a vector space can define a Banach space, and proving Theorem \ref{theorem1.1}. 

Finally, throughout this text the symbol $\mathbb{K}$ will denote, indistinctly, the field of real numbers $\mathbb{R}$ or the field $\mathbb{C}$ of complex numbers. From this point on, unless otherwise stated, all spaces considered, including $ E $ and any other relevant spaces, will be assumed to be vector spaces over $ \mathbb{K} $.


\section{Norm induced by a bijective linear map}\label{Sec2}

The main goal of this section is to explore some questions about norms in vector spaces that are induced by bijective linear maps. In particular, we aim to prove Theorem \ref{theorem1.2}. Given a vector space $E$ over $\mathbb{K}$, a natural approach to defining a norm on $E$ that preserves certain topological structures is through isometries. More precisely, consider the following result:

\begin{lemma}\label{lemma3.1}  
	Let $E$ be a vector space and $(F, \| \cdot \|_F)$ a normed space. Suppose there exists a bijective linear operator $T: E \to F$, and consider the function $\| \cdot \|: E \to \mathbb{R}$ defined by  
	\begin{equation*}\label{T-norm} 
		\| u \|= \| Tu \|_F  \quad \text{for all } u \in E.
	\end{equation*}
	Then:  
	\begin{enumerate}[label=(\alph*)]
		\item \label{2.1a} $\| \cdot \|$ is a norm on $E$; 
		\item \label{2.1b} If $\| \cdot \|_F$ is induced by an inner product, then $\| \cdot \|$ is also induced by an inner product; 
		\item \label{2.1c} If $F$ is a Banach space, then $(E, \| \cdot \|)$ is also a Banach space;  
		\item \label{2.1d} If $F$ is not separable, then $(E, \| \cdot \|)$ is not separable;  
		\item \label{2.1e} If $F$ is separable, then $(E, \|\cdot\|)$ is separable.  
	\end{enumerate}
	Now, if we endow $E$ with another norm $\| \cdot \|_1$ and consider $T: (E, \| \cdot \|_1) \to (F, \| \cdot \|_F)$ as a linear operator between normed spaces, we have:  
	\begin{enumerate}[label=(\alph*), start=6]  	 
		\item \label{2.1f} $\|\cdot \|_1 \sim \|\cdot \|$ if and only if $T$ and $T^{-1}$ are continuous;  
		\item \label{2.1g} $\|\cdot \|_1 \not\sim \|\cdot \|$ if and only if $T$ or $T^{-1}$ is non-continuous.  
	\end{enumerate}
\end{lemma}  
\begin{proof}
To prove item \ref{2.1a}, note that, due to the injectivity of $ T $, $ \text{ker}(T) = \{0\} $. Thus,  
$$
\| u \| = 0 \iff \|Tu\|_F = 0 \iff Tu = 0 \iff u = 0.
$$  
Moreover, due to the linearity of $ T $ and the fact that $\|\cdot\|_F$ is a norm, we obtain the properties  
$$
\|\lambda u\| = |\lambda| \|u\|  
\quad \text{and} \quad  
\|u + v\| \leq \|u\| + \|v\|, \quad \forall\,\, u,v \in E, \; \lambda \in \mathbb{K}.
$$  
Thereby, $\|\cdot\|$ defines a norm on $ E $. To prove item \ref{2.1b}, suppose that $\|\cdot\|_F$ is induced by the inner product $(\cdot,\cdot)_F$. In this case, the function  
$$
(u, v) = (Tu, Tv)_F, \quad \forall\,\, u,v \in E,
$$  
defines an inner product on $ E $, which induces the norm $\|\cdot\|$. For item \ref{2.1c}, consider $(u_n)$ a Cauchy sequence in $(E, \|\cdot\|)$. Since  
$$
\|Tu_n - Tu_m\|_F = \|u_n - u_m\| \to 0 \quad \text{as} \quad m, n \to +\infty,
$$  
it follows that $(Tu_n)$ is a Cauchy sequence in $ F $. By the completeness of $ F $, there is $ v \in F $ such that $ Tu_n \to v $ in $ F $. Moreover, by the surjectivity of $ T $, there exists $ u \in E $ such that $ Tu = v $. Consequently,  
$$
\|u_n - u\| = \|Tu_n - v\|_F \to 0 \quad \text{as} \quad n \to +\infty.
$$  
Therefore, $(E, \|\cdot\|)$ is a Banach space. For item \ref{2.1d}, suppose, for contradiction, that $(E, \|\cdot\|)$ is separable. This implies the existence of a countable subset $ A \subset E $ that is dense in $ E $ with respect to the norm $\|\cdot\|$. We will prove that the countable set $ T(A) \subset F $ is also dense in $ F $. Let $ v \in F $ and $\varepsilon > 0$. By the surjectivity of $ T $, there exists $ u \in E $ such that $ v = Tu $. By the density of $ A $, there exists $ w \in A $ such that $\|u - w\| < \varepsilon$. Thus,  
$$
\|v - Tw\|_F = \|Tu - Tw\|_F = \|u - w\| < \varepsilon.
$$  
This shows that $ T(A) $ is dense in $ F $, contradicting the assumption that $ F $ is not separable. For item \ref{2.1e}, suppose that $ F $ is separable. Then, by the surjectivity of $ T $, there exists a countable set $ A\subset E $ such that $ T(A) $ is dense in $ F $. Therefore, for each $ u\in E $ and $\varepsilon>0$, by density, there exists $ w\in A $ such that $ \|Tu - Tw\|_F<\varepsilon $, or equivalently, $ \|u-w\|<\varepsilon $, showing that $ A $ is dense in $ E $. Hence, $ E $ is separable. Item \ref{2.1f} follows immediately because $ T $ and $ T^{-1} $ are continuous if and only if there exist constants $ c_1,c_2>0 $ such that  
$$
c_1 \| u \|_1 \leq \| Tu \|_F \leq c_2 \| u \|_1 \quad \forall u\in E,  
$$  
or equivalently, $ \|\cdot\|_1\sim \|\cdot\| $. Finally, by the contrapositive of item \ref{2.1f}, we obtain item \ref{2.1g}.	
\end{proof}

Note that, by construction, the space $(E, \|\cdot\|)$ considered in Lemma \ref{lemma3.1} is isometric to the space $(F, \|\cdot\|_F)$. Consequently, it inherits all topological properties of $F$ that are invariant under isometries, such as reflexivity.

\begin{proof}[{\bf Proof of Theorem \ref{theorem1.2}.}]
The result follows directly from item \ref{2.1c} of Lemma \ref{lemma3.1}.	
\end{proof}

Let us now look at some applications of Lemma \ref{lemma3.1}.

\begin{example}
Let us show that the following real vector space of functions
$$
E = \big\{ f \in C^1([0,1], \mathbb{R}) : f(0) = 0 \big\},
$$
equipped with the maximum norm of the derivative $ \|f\| = \|f'\|_{\infty} $, is a Banach space. To do this, consider the bijective linear operator
$$
T : (E, \| \cdot \|_{\infty}) \longrightarrow (C([0,1], \mathbb{R}), \|\cdot\|_\infty ), \quad f \longmapsto f',
$$
with the inverse given by
$$
T^{-1}f(t) = \int_0^t f(s) \, \mathrm{d}s.
$$
Now consider the sequence of functions $ f_n(t) =\dfrac{t^n}{n}$ in $ E $. For this sequence, we have
$$
\frac{ \| Tf_n \|_\infty } { \| f_n \|_\infty } = n \to+ \infty,
$$
which implies that the operator $ T $ is non-continuous. Based on Lemma \ref{lemma3.1}, using items \ref{2.1a}, \ref{2.1c}, and \ref{2.1g}, we can conclude the following:
\begin{enumerate}
	\item[1.] $\| f \| = \| Tf \|_\infty = \| f' \|_\infty$ is a norm on $ E $;
	\item[2.] The space $ (E, \|\cdot\|) $ is Banach;
	\item[3.] The norms $ \|\cdot\|_\infty $ and $ \|\cdot\| $ are not equivalent on $ E $.
\end{enumerate}	
\end{example} 

We know that when two norms $ \|\cdot\|_1 $ and $ \|\cdot\|_2 $ on a vector space $ E $ are equivalent, then every Cauchy sequence with respect to the norm $ \|\cdot\|_1 $ is also a Cauchy sequence with respect to the norm $ \|\cdot\|_2 $, and vice versa. A natural question that arises is whether the converse of this statement is also true, that is, does the fact that the norms $ \|\cdot\|_1 $ and $ \|\cdot\|_2 $ preserve Cauchy sequences imply that they are equivalent? The answer to this question is affirmative, as will be proven in the following result.

\begin{corollary} 
	Norms on a normed vector space $ E $ are equivalent if and only if Cauchy sequences in $ E $ are preserved between them.
\end{corollary}
\begin{proof}
It suffices to show that if Cauchy sequences in $ E $ are preserved between the norms $ \|\cdot\|_1 $ and $ \|\cdot\|_2 $, then these norms are equivalent. In other words, according to Lemma \ref{lemma3.1}-\ref{2.1f}, we need to prove that the identity operator $ I: (E, \|\cdot\|_1) \to (E, \|\cdot\|_2) $ is continuous and has a continuous inverse, since $ \|\cdot\|_2 = \|I(\cdot)\|_2 $. Assume, for contradiction, that the identity operator $ I $ is not continuous. This means there exists a sequence $ (u_n) $ in $ E $ such that  
\begin{equation}\label{4} 
	\|u_n\|_2 > 2^n \|u_n\|_1,\quad \forall\,\, n\in\mathbb{N}.	
\end{equation}
Now, define the following sequence
$$
v_n = \sum_{k=1}^n \frac{1}{2^k} \frac{u_k}{\|u_k\|_1}.
$$
For $ n > m $ and applying the triangle inequality, we obtain
$$
\|v_n - v_m\|_1 = \left\|\sum_{k=m+1}^n \frac{1}{2^k} \frac{u_k}{\|u_k\|_1}\right\|_1 \leq  \sum_{k=m+1}^n \frac{1}{2^k}.
$$
From this, it follows that $ (v_n) $ is a Cauchy sequence with respect to the norm $ \|\cdot\|_1 $. On the other hand, from \eqref{4}, we deduce 
$$
\|v_n - v_{n-1}\|_2 = \left\|\frac{1}{2^n} \frac{u_n}{\|u_n\|_1}\right\|_2 = \frac{\|u_n\|_2}{2^n\|u_n\|_1} > 1, \quad \forall\,\, n \in \mathbb{N}.
$$
This implies that $ (v_n) $ is not a Cauchy sequence with respect to the norm $ \|\cdot\|_2 $, which is a contradiction, as it contradicts the assumption that Cauchy sequences in $ \|\cdot\|_1 $ and $ \|\cdot\|_2 $ are mutually preserved. Therefore, the identity operator $ I $ must be continuous. The same reasoning shows that $ I:(E, \|\cdot\|_1) \to (E, \|\cdot\|_2) $ also has a continuous inverse. This completes the proof.
\end{proof}


\section{Banach spaces with two non-equivalent norms}\label{Sec3}

In this section, we will discuss the proofs of Theorems \ref{theorem1.3}, \ref{theorem1.4} and \ref{NewTheorem}. We saw in the previous section that Lemma \ref{lemma3.1} presents a way to define norms on a vector space $ E $. This approach is based on the possibility of finding a normed space $ (F, \|\cdot\|_F) $ and a bijective linear operator $ T: E \to F $. In this scenario, the norm generated by $ T $ preserves the topological properties of the normed space $ (F, \|\cdot\|_F) $, ensuring, for example, completeness, separability and reflexivity. However, if $ E $ is already equipped with a norm $ \|\cdot\|_1 $ and the operator $ T: (E, \|\cdot\|_1) \to (F, \|\cdot\|_F) $ is not continuous, then the norm induced by $ T $ will not be equivalent to the original norm $ \|\cdot\|_1 $. This observation highlights that the relationship between norms and operators is directly connected to continuity. Based on this idea, we present the proof of Theorem \ref{theorem1.3}.

\begin{proof}[{\bf Proof of Theorem \ref{theorem1.3}.}]
Since the dimension of $ E $ is infinite, there exists a Hamel basis $ \mathcal{B} $ consisting of infinitely many elements. Consider, then, a countably infinite subset $ \mathcal{B}' = \{u_n\}_{n \in \mathbb{N}} \subset \mathcal{B} $. By Proposition 2.13 of \cite{Knapp}, there exists a unique linear functional $ \varphi: E \to \mathbb{K} $ such that  
$$
\varphi(u_n) = n\|u_n\|_1 \quad \text{and} \quad \varphi(u) = 0 \quad \text{for all } u \in \mathcal{B} \setminus \mathcal{B}'.
$$
Now, we note that for $ e = u_2/\|u_2\|_1 $, one has $ \varphi(e) = 2 $. Based on this, we define the operator $ T: (E, \|\cdot\|_1) \to (E, \|\cdot\|_1) $ by  
$$
Tu = u - \varphi(u)e.
$$
We claim that $ T $ is linear. Indeed, for $ u, v \in E $ and $ \lambda \in \mathbb{K} $, we have  
$$
T(u + \lambda v) = (u + \lambda v) - \varphi(u + \lambda v)e = \big(u - \varphi(u)e\big) + \lambda\big(v - \varphi(v)e\big) = Tu + \lambda Tv.
$$
Moreover, $ T $ is bijective, and its inverse is given by $ T^{-1} = T $, meaning that $ T^2 = I $, where $ I $ is the identity operator. Indeed, we observe that $ Te = -e $ and, for any $ u \in E $,  
$$
T(Tu) = T\big(u - \varphi(u)e\big) = Tu - \varphi(u)Te = \big(u - \varphi(u)e\big) + \varphi(u)e = u.
$$
We also show that $ T $ is non-continuous because 
$$
\frac{\|Tu_n\|_1}{\|u_n\|_1} = \frac{\|\varphi(u_n)e - u_n\|_1}{\|u_n\|_1} \geq \frac{\varphi(u_n)\|e\|_1 - \|u_n\|_1}{\|u_n\|_1} = n - 1,
$$
and so,
$$
\frac{\|Tu_n\|_1}{\|u_n\|_1} \to +\infty \quad \text{as } n \to+ \infty.
$$ 
Finally, we define a new norm $ \|u\|_2 = \|Tu\|_1 $ on $ E $. Applying Lemma \ref{lemma3.1}-\ref{2.1g}, we conclude that the norms $ \|\cdot\|_1 $ and $ \|\cdot\|_2 $ are not equivalent, since $ T = T^{-1} $ is non-continuous. However, the space $ (E, \|\cdot\|_2) $ is Banach, as $ (E, \|\cdot\|_1) $ is also Banach. Thus, the theorem is proved.	
\end{proof}

\begin{remark}\label{remark3.4}  
	In the construction of the operator $ T $ presented in the proof of Theorem \ref{theorem1.3}, we did not use the fact that $ (E, \|\cdot\|_1) $ is a Banach space. This condition was only employed to ensure that $ (E, \|\cdot\|_2) $ is also a Banach space. Thus, we can conclude that in any infinite-dimensional real or complex normed space $ (E, \|\cdot\|_1) $, it is always possible to construct a linear operator $ T: E \to E $ that is bijective, non-continuous, and has an non-continuous inverse.  
\end{remark}

\begin{corollary}\label{corollary3.1}  
	If $ (E, \|\cdot\|_1) $ is an infinite-dimensional Hilbert space, then there exists a norm $ \|\cdot\|_2 $ such that $ (E, \|\cdot\|_2) $ is also a Hilbert space, and the norms $ \|\cdot\|_1 $ and $ \|\cdot\|_2 $ are not equivalent.  
\end{corollary}  
\begin{proof}  
According to Remark \ref{remark3.4}, we can ensure the existence of a linear operator $ T: (E, \|\cdot\|_1) \to (E, \|\cdot\|_1) $ that is bijective and non-continuous. Based on Lemma \ref{lemma3.1}-\ref{2.1b}-\ref{2.1c}-\ref{2.1g}, we conclude that the norms $ \|\cdot\|_1 $ and $ \|\cdot\|_2 $ are not equivalent in $ E $, where $ \|u\|_2 = \|Tu\|_1 $ for $u\in E$. Furthermore, since $ (E, \|\cdot\|_1) $ is a Hilbert space and $ T $ is a linear bijection, the space $ (E, \|\cdot\|_2) $ is Banach, and the norm $ \|\cdot\|_2 $ is induced by an inner product, completing the proof.  
\end{proof}

\begin{corollary}\label{corollary3.2}  
	If $ (E, \|\cdot\|_1) $ is an infinite-dimensional non-separable normed space, then there exists a norm $ \|\cdot\|_2 $ such that $ (E, \|\cdot\|_2) $ is also non-separable, and the norms $ \|\cdot\|_1 $ and $ \|\cdot\|_2 $ are not equivalent.
\end{corollary}
\begin{proof}
By Remark \ref{remark3.4}, there is a linear operator $ T: (E, \|\cdot\|_1) \to (E, \|\cdot\|_1) $, which is bijective and non-continuous. From Lemma \ref{lemma3.1}-\ref{2.1d}-\ref{2.1g}, we can conclude that the norm $ \|u\|_2 = \|Tu\|_1 $ is not equivalent to the norm $ \|\cdot\|_1 $. Furthermore, as $ (E, \|\cdot\|_1) $ is not separable and the operator $ T $ is a linear bijection, the space $ (E, \|\cdot\|_2) $ will also be non-separable.
\end{proof}

Arguing in the same way as in the proof of the Corollary \ref{corollary3.2}, we can conclude that:  

\begin{corollary} \label{corollary3.3} 
	If $ (E, \|\cdot\|_1) $ is an infinite-dimensional separable normed space, then there exists a norm $ \|\cdot\|_2 $ such that $ (E, \|\cdot\|_2) $ is also separable, and the norms $ \|\cdot\|_1 $ and $ \|\cdot\|_2 $ are not equivalent.  
\end{corollary}

From this last result, the following question arises: {\it Is there a Banach space with two non-equivalent norms such that it is separable with one of them but not with the other?} To answer this question, let us consider a separable Banach space $ (E,\|\cdot\|_1) $ over $ \mathbb{K} $. Thanks to Corollary 6.8 in \cite{Carothers}, we know that $ E $ is isometrically isomorphic to a subspace of $ (\ell^\infty, \|\cdot\|_{\infty}) $, and thus, $ {\rm card}(E) \leq {\rm card}(\ell^\infty) = \mathfrak{c} $. Here, $\| \cdot\|_{\infty}$ denotes the supremum norm in $\ell^{\infty}$. On the other hand, since $ E $ is Banach, it follows from Lemma 3.2-(b) of \cite{Lorenz} that $ {\rm card}(E) \geq \mathfrak{c} $. This shows that every separable Banach space has the cardinality of the continuum. With this in mind, we can proceed to prove Theorem \ref{theorem1.4}.

\begin{proof}[{\bf Proof of Theorem \ref{theorem1.4}.}]
Let $(E, \|\cdot\|_1)$ be a separable Banach space of infinite dimension. As discussed earlier, we know that $\text{card}(E) = \mathfrak{c}$. By Theorem 3.5 of \cite{Lorenz}, any Hamel basis of a real or complex Banach space of infinite dimension has the same cardinality as the space. Therefore, if $\mathcal{B}$ and $\mathcal{B}'$ are Hamel bases for $E$ and $\ell^\infty$, respectively, then 
$$
\text{card}(\mathcal{B}) = \text{card}(E) = \mathfrak{c}
\qquad \text{and} \qquad 
\text{card}(\mathcal{B}') = \text{card}(\ell^\infty) = \mathfrak{c}.
$$
Thus, $\text{card}(\mathcal{B}) = \text{card}(\mathcal{B}')$, that is, there exists a bijective function from $\mathcal{B}$ to $\mathcal{B}'$. This allows us to arbitrarily assign to each vector $u \in \mathcal{B}$ a unique vector $f_u \in \mathcal{B}'$. Consequently, by Proposition 2.13 of \cite{Knapp}, there exists a unique linear operator $T: E \to (\ell^\infty, \|\cdot\|_\infty)$ such that 
$$
Tu = f_u, \quad \forall \ u \in \mathcal{B}, \quad \text{where } f_u \in \mathcal{B}'.
$$
Let us now verify that $T$ is bijective. To prove that $T$ is injective, let $u, v \in E$ be such that $Tu = Tv$. Thereby, we can write $u = a_1u_1 + \cdots + a_nu_n$ and $v = b_1u_1 + \cdots + b_nu_n$, where $a_j, b_j$ are scalars and $u_j \in \mathcal{B}$. Thus,
$$
a_1 Tu_1 + \cdots + a_n Tu_n = b_1 Tu_1 + \cdots + b_n Tu_n.
$$
However, since $\{Tu_1, \ldots, Tu_n\} \subset \mathcal{B}'$ is linearly independent, we obtain $a_j = b_j$ for all $j = 1, \ldots, n$, showing that $u = v$, which establishes the injectivity of $T$. On the other hand, given $g \in \ell^\infty$, there exist scalars $a_1, \ldots, a_n$ and vectors $f_1, \ldots, f_n \in \mathcal{B}'$ such that $g = a_1f_1 + \cdots + a_nf_n$. By the definition of $T$, there exist $u_1, \ldots, u_n \in \mathcal{B}$ such that $Tu_i = f_i$ for all $i = 1, \ldots, n$. Consequently,
$$
g = a_1f_1 + \cdots + a_nf_n = a_1 Tu_1 + \cdots + a_n Tu_n = T(a_1u_1 + \cdots + a_nu_n),
$$
which proves that $T$ is surjective. Therefore, since $T$ is a bijective linear operator, based on Lemma \ref{lemma3.1}-\ref{2.1a}-\ref{2.1c}-\ref{2.1d}, we obtain:
\begin{enumerate}
	\item[1.] $\|u\|_2 := \|Tu\|_{\infty}$ defines a norm on $E$;
	
	\item[2.] $(E, \|\cdot\|_2)$ is a Banach space, since $\ell^\infty$ is a Banach space;
	
	\item[3.] $(E, \|\cdot\|_2)$ is not separable, since $\ell^\infty$ is not separable.
\end{enumerate}
Finally, to prove the non-equivalence of the norms assume for contradiction that there exists a constant $c > 0$ such that
$$
\|u\|_1 \leq c\|u\|_2, \quad \forall\,\, u \in E.
$$
This shows that the identity operator $I: (E, \|\cdot\|_2) \to (E, \|\cdot\|_1)$ between Banach spaces is continuous. Hence, by Corollary 2.8 in \cite{Brezis}, we conclude that the norms $\|\cdot\|_1$ and $\|\cdot\|_2$ are equivalent. However, as $(E, \|\cdot\|_1)$ is separable, $(E, \|\cdot\|_2)$ should also be separable, but this is a contradiction. Therefore, $ \|\cdot\|_1 $ and $ \|\cdot\|_2 $ are not equivalent. This completes the proof.	
\end{proof}

\begin{remark}
	Note that the operator $ T $ constructed in the proof of Theorem \ref{theorem1.4} has the property that $ T $ or $ T^{-1} $ cannot be continuous when $E$ is equipped with the norm $\|\cdot \|_1$, as stated in Lemma \ref{lemma3.1}, item \ref{2.1g}.
\end{remark}

In contrast with Corollary \ref{corollary3.1}, we have the following result:

\begin{corollary}
	For every separable infinite-dimensional Hilbert space $(E, \|\cdot\|_1)$, there exists another norm $\|\cdot\|_2$ that is not induced by an inner product, such that $(E, \|\cdot\|_2)$ is a non-separable Banach space and $\|\cdot\|_1 \not\sim \|\cdot\|_2$.
\end{corollary}
\begin{proof}
From the proof of Theorem \ref{theorem1.4}, we know that there exists a bijective linear operator $T: E \to \ell^\infty$ such that $\|u\|_2 = \|Tu\|_\infty$ defines a norm on $E$ that is not equivalent to $\|\cdot\|_1$, and $(E, \|\cdot\|_2)$ is a Banach and non-separable space. Finally, to show that $\|\cdot\|_2$ is not induced by an inner product, assume, by contradiction, that it is, i.e., that $\|\cdot\|_2$ satisfies the parallelogram identity \cite{Jordan}. Thus, for each $w, z \in \ell^\infty$, there exist $u, v \in E$ such that $Tu = w$ and $Tv = z$. Consequently,  
$$  
\begin{aligned}  
	\|w+z\|_{\infty}^{2} + \|w-z\|_{\infty}^2  
	&= \|T(u+v)\|_{\infty}^2 + \|T(u-v)\|_{\infty}^2 \\  
	&= \|u+v\|_{2}^2 + \|u-v\|_{2}^2 \\  
	&= 2\left( \|u\|_{2}^2 + \|v\|_{2}^2 \right) \\  
	&= 2\left( \|w\|_{\infty}^2 + \|z\|_{\infty}^2 \right)  
\end{aligned}  
$$  
which shows that $\|\cdot\|_{\infty}$ satisfies the parallelogram identity, a contradiction.  
\end{proof}

We will now present another consequence of Theorem \ref{theorem1.4}.

\begin{corollary}
	Let $(E, \|\cdot\|_1)$ be a separable Banach space. Among the unit spheres in $E$, there exists at least one that is separable and one that is not.
\end{corollary}
\begin{proof}
Since $(E, \|\cdot\|_1)$ is a separable Banach space, by Theorem \ref{theorem1.4}, there exists a norm $\|\cdot\|_2: E \to \mathbb{K}$ such that $\| \cdot \|_1 \not\sim \| \cdot \|_2$ and $(E, \|\cdot\|_2)$ is a non-separable Banach space. Now, consider the unit spheres 
$$
S_1 = \{x \in E : \|x\|_1 = 1\} \quad \text{and} \quad S_2 = \{x \in E : \|x\|_2 = 1\}.
$$
Note that $S_1$ is separable, because $(E, \|\cdot\|_1)$ is separable. On the other hand, we claim that $S_2$ is non-separable. To prove this by contradiction, assume there exists a countable dense subset $A$ of $S_2$. Now, we consider the following subset of $E$:
$$
D= \left\{ta \in E \mid t \in \mathbb{Q} \text{ and } a \in A\right\},
$$
where $\mathbb{Q}$ is the set of rational numbers. Since $A$ and $\mathbb{Q}$ are countable, we have that $D$ is also countable. Now, let us show that $\overline{D}=E$. Indeed, for any $x \in E \setminus \{0\}$, there exist sequences $(t_n)$ in $\mathbb{Q}$ and $(a_n)$ in $A$ such that
\begin{equation}\label{NewEquation1}
	t_n \to \|x\|_2 \text{ in } \mathbb{R} \quad \text{and} \quad a_n \to \frac{x}{\|x\|_2} \text{ in } E,
\end{equation}
since $\overline{\mathbb{Q}} = \mathbb{R}$ and $\overline{A} = S_2$. We observe that $(t_n a_n)\in D$ for all $n\in\mathbb{N}$ and
\begin{equation}\label{NewEquation2}
	\|t_na_n-x\|_2\le\left|t_n -\|x\|_2\right|\|a_n\|_2+\left\|\|x\|_2a_n-x\right\|_2.
\end{equation}
Combining \eqref{NewEquation1} and \eqref{NewEquation2}, we deduce that $t_n a_n \to x$ in $E$. Therefore, we conclude from this convergence that $\overline{D} = E$, meaning that $(E, \|\cdot\|_2)$ is separable, which is a contradiction, and the result follows.
\end{proof}

The idea of the proof of Theorem \ref{theorem1.4} motivates the proof of Theorem \ref{NewTheorem}. Now, before presenting the proof of Theorem \ref{NewTheorem}, it is useful to introduce the following: we say that a property $P$ is \textit{preserved under equivalent norms} if, for two normed spaces $(E, \|\cdot\|_1)$ and $(E, \|\cdot\|_2)$, the norms $\|\cdot\|_1$ and $\|\cdot\|_2$ are equivalent, and $E$ has the property $P$ with respect to the norm $\|\cdot\|_1$ if and only if $E$ also has it with respect to the norm $\|\cdot\|_2$. Moreover, we say that $P$ is \textit{preserved under isometries} if, for two normed spaces $(E, \|\cdot\|_E)$ and $(F, \|\cdot\|_F)$, and a bijective isometry $T: E \to F$, if $E$ has the property $P$ with respect to $\|\cdot\|_E$, then $F$ also has the property $P$ with respect to $\|\cdot\|_F$. Below are examples of properties $ P $ that are invariant under equivalent norms and under isometries in Banach spaces:

\begin{enumerate}  
	\item[1.] \textbf{Reflexivity:} A Banach space $ E $ is reflexive if the natural embedding $ J_E: E \to E'' $ is a continuous isomorphism.   
	
	\item[2.] \textbf{Schur Property:} A Banach space $ E $ has the Schur property if every weakly convergent sequence in $ E $ is norm-convergent. 
	
	\item[3.] \textbf{Approximation Property (AP):} A Banach space $ E $ has the AP if the identity operator on $ E $ can be approximated uniformly on compact sets by finite-rank operators. 
	
	\item[4.] \textbf{Banach-Saks Property:} A Banach space $ E $ has the Banach-Saks property if, for any sequence $ (x_n) $ in $ E $, there exists a subsequence $ (x_{n_k}) $ such that the sequence of its arithmetic means
	$$
	y_N = \frac{1}{N} \sum_{k=1}^N x_{n_k}
	$$
	converges strongly (in norm) to some element $ y \in E $.
	
	\item[5.] \textbf{Dunford-Pettis Property:} A Banach space $ E $ has the Dunford-Pettis property if every continuous weakly compact operator $ T: E \to F $ from $ E$ into another Banach space $ F $ transforms weakly compact sets in $ E $ into norm-compact sets in $ F $.  
\end{enumerate} 

We are now in a position to prove Theorem \ref{NewTheorem}. 

\begin{proof}[{\bf Proof of Theorem \ref{NewTheorem}.}]
Let $(E, \|\cdot\|_1)$ be a Banach space that satisfies a property $ P $ that is preserved under equivalent norms, and let $ (F,\|\cdot\|_F )$ be a Banach space with the same cardinality as $ E $ that does not satisfy $ P $. Since $ E $ and $ F $ have the same cardinality, from \cite[Theorem 3.5]{Lorenz}, any Hamel basis $\mathcal{B}$ of $E$ and any Hamel basis $\mathcal{B}'$ of $F$ have the same cardinality. Consequently, we can argue as in Theorem \ref{theorem1.4} to deduce the existence of a bijective linear operator $ T: E \to F $. We then define a norm $ \|\cdot\|_2 $ on $ E $ by 
$$
\|x\|_2 = \|Tx\|_F,\quad x\in E.
$$ 
Since $ T $ is an isometry, the space $ (E, \|\cdot\|_2) $ is a Banach space. Now, assume that there exists a constant $ c > 0 $ such that $ \|x\|_1 \leq c\|x\|_2 $ for all $ x \in E $. This implies that the norms $ \|\cdot\|_1 $ and $ \|\cdot\|_2 $ are equivalent. Since $ (E, \|\cdot\|_1) $ satisfies the property $ P $ and $ P $ is preserved under equivalent norms, it follows that $ (E, \|\cdot\|_2) $ also satisfies $ P $. However, since $ (E, \|\cdot\|_2) $ is isometric to $ (F, \|\cdot\|_F) $, this would imply that $ F $ satisfies $ P $, since $P$ is invariant under bijective isometries, which contradicts our assumption. Therefore, there cannot exist a constant $ c > 0 $ such that $ \|x\|_1 \leq c\|x\|_2 $ for all $ x \in E $. Hence, the norms $ \|\cdot\|_1 $ and $ \|\cdot\|_2 $ are not equivalent, and by isometry, the space $ (E, \|\cdot\|_2) $ does not satisfy the property $ P $.	This completes the proof.
\end{proof}
	
\begin{corollary}  
	If $(E, \|\cdot\|_1)$ is a reflexive Banach space with the cardinality of the continuum $\mathfrak{c}$, then there exists a norm $\|\cdot\|_2$ that is not equivalent to $\|\cdot\|_1$ and such that $(E, \|\cdot\|_2)$ is not reflexive.  
\end{corollary}  
\begin{proof}  
It suffices to consider the space $\ell^\infty$, since $\ell^\infty$ is non-reflexive and has the cardinality of the continuum. Therefore, applying Theorem \ref{NewTheorem}, there exists a norm $\|\cdot\|_2$ on $E$ such that $\|\cdot\|_2$ is not equivalent to $\|\cdot\|_1$ and $(E, \|\cdot\|_2)$ is not reflexive.  
\end{proof}  

To conclude this section, we would like to highlight that not every property $P$ is preserved under equivalent norms. A notable example occurs in $\mathbb{R}^2$, where the Euclidean norm $\|\cdot\|_e$ and the sum norm $\|\cdot\|_s$ are equivalent, but only the former guarantees that the space is uniformly convex. Furthermore, the Theorem \ref{NewTheorem} presents some limitations, such as the assumption of the existence of a Banach space $F$ with the same cardinality as $E$ that does not satisfy $P$. Although it is possible to guarantee that for each cardinality $\kappa$ there exists a vector space with a Hamel basis of size \(\kappa\), the construction of Banach spaces with specific structures for each \(\kappa \geq \mathfrak{c}\) remains an open question. In particular, when \(\kappa = \aleph_0\), we know from Baire’s Theorem that an infinite Banach space cannot have a countable Hamel basis. Thus, Theorem \ref{NewTheorem} does not apply in the case $\kappa = \aleph_0$.


\section{Vector spaces with the $p$-norm}\label{Sec4}

The main objective of this section is to present a proof of Theorem \ref{theorem1.1}. Following the reasoning of the previous sections, to prove it, it suffices to find a non-complete normed space $(F,\|\cdot \|_F)$ that is isometric to $(E, \|\cdot\|_p)$, where $\|\cdot\|_p$ is the $p$-norm defined in \eqref{p-norm}. To do this, note that, given a Hamel basis for $E$, each element $u$ has a finite number of nonzero coordinates. This allows us to identify it as a function with finite support, that is, a function that takes a finite number of nonzero values. With this concept in mind, we now introduce a class of spaces that will enable us to solve the problem. Let $ I $ be a nonempty set and $ \mu $ be the counting measure on $ I $, defined on the $ \sigma $-algebra of all subsets of $ I $ (see \cite{Axler,Jones}). For each $ 1\leq p<\infty $, the space $ \ell^p(I) $ is defined as  
$$
\ell^p(I) = \left\{ u : I\to \mathbb{K} : \int_{I} |u|^p \,  \mathrm{d}\mu <\infty  \right\},  
$$
and the space $ \ell^{\infty}(I) $ as  
$$
\ell^{\infty}(I) = \left\{ u:I \to \mathbb{K}: \sup_{k\in I}|u(k)|<\infty \right\}.  
$$
Since the only set with zero measure is the empty set, we obtain that $\ell^p(I)$ is the Lebesgue space $L^p(I,2^{I},\mu)$, where $2^{I}$ is the set of parts of $I$. Thus, the space $ \ell^p(I) $, equipped with the norm  
$$
\|u\|_{\ell^p} = \left( \int_{I} |u|^{p} \, \mathrm{d}\mu \right)^{1/p}  
\quad \text{and} \quad  
\|u\|_{\ell^{\infty}} = \sup_{k \in I} |u(k)|,  
$$
is a Banach space. Now, consider the space of functions with finite support, defined as  
$$
C_{c}(I)=\Big\{u:I\to \mathbb{K}: u(k)\ne 0 \; \text{for a finite number of indices } k\in I   \Big\}.
$$
Regarding this space, we establish the following properties:

\begin{lemma}\label{lemma4.1}  
	Let $ I $ be an infinite set and $ 1 \leq p < \infty $. Then:  
	\begin{enumerate}[label=(\alph*)]  
		\item \label{4.1a} $ (C_{c}(I) \|\cdot\|_{\ell^p}) $ is not a Banach space;  
		\item \label{4.1b} $ (C_{c}(I), \|\cdot\|_{\ell^p}) $ is separable if and only if $ I $ is countable.  
	\end{enumerate}  
\end{lemma}  
\begin{proof}  
To prove \ref{4.1a}, consider a subset $ \{k_m\}_{m\in\mathbb{N}} $ of $ I $ and define the functions  
$$  
u = \sum_{m=1}^{\infty} m^{\alpha} \, \chi_{\{k_m\}}  
\qquad \text{and} \qquad  
u_{n} = \sum_{m=1}^{n} m^{\alpha} \, \chi_{\{k_m\}},  
$$  
where $\alpha = \dfrac{-2}{p}$. By definition, note that $ u \notin C_{c}(I) $ and $ (u_n) \subset C_{c}(I) $. Moreover,  
$$  
\|u - u_n\|_{\ell^p}^p  
= \int_{I} \left| \sum_{m=n+1}^{\infty} m^{\alpha} \, \chi_{\{k_m\}} \right|^p \, \mathrm{d}\mu  
= \int_{I} \sum_{m=n+1}^{\infty} m^{-2} \, \chi_{\{k_m\}} \, \mathrm{d}\mu  
= \sum_{m=n+1}^{\infty} \frac{1}{m^2}.  
$$  
Taking the limit as $ n \to \infty $, we obtain $ u_n \to u $ in $ \ell^p(I) $. This shows that $ C_{c}(I) $ is not a closed subspace of $ \ell^p(I) $, and thus $ (C_{c}(I), \|\cdot\|_{\ell^p}) $ cannot be a Banach space. To prove item \ref{4.1b}, assume first that $ C_c(I) $ is separable and suppose, for contradiction, that $ I $ is uncountable. Then, the set of characteristic functions $ \{\chi_{\{k\}}\}_{k \in I} \subset C_{c}(I) $ is uncountable, and  
$$  
\|\chi_{\{k\}} - \chi_{\{n\}}\|_{\ell^p}^p = \int_{I} |\chi_{\{k\}} - \chi_{\{n\}}|^p \, \mathrm{d}\mu = \int_{I} \chi_{\{k\}} + \chi_{\{n\}} \, \mathrm{d}\mu = 2 \quad \forall\,\, k, n \in I, \; k \ne n.  
$$  
Thus, $ C_{c}(I) $ cannot be separable, which is a contradiction. Conversely, without loss of generality, assume $ I = \mathbb{N} $. In this case, we can write  
$$  
C_c(I) = \bigcup_{k \in \mathbb{N}} E_k \quad \text{where} \quad E_k = \mathrm{span}\{\chi_{\{1\}}, \ldots, \chi_{\{k\}} \}.  
$$  
Since every finite-dimensional normed space is separable, $ (E_k, \|\cdot\|_{\ell^p}) $ is separable. Therefore, the countable union of these subspaces is also separable. Thereby, we conclude that $ (C_c(I), \|\cdot\|_{\ell^p}) $ is separable. This completes the proof of \ref{4.1b}.  
\end{proof}  

\begin{proof}[{\bf Proof of Theorem \ref{theorem1.1} for the case $1\leq p<\infty$.}]
First, suppose that $ \dim(E) < \infty $. In this case, we have that $ E $ is a Banach space. Now, assume that $ (E, \|\cdot\|_p) $ is a Banach space and, for the sake of contradiction, suppose that $ \dim(E) = \infty $. Thereby, there exists a basis $ \mathcal{B} = \{e_k\}_{k \in I} $ of $ E $ with infinitely many elements, where $I$ denotes the index set. We then define the function $ T: C_c(I) \to E $ by  
$$
Tu = \sum_{k \in J} u(k) e_k, 
\quad \text{where} \quad 
u = \sum_{k \in J} u(k) \chi_{\{k\}},
$$
where $ J \subset I $ is finite. The function $ T $ is linear since, for every $ u, v \in C_c(I) $ and $ \lambda \in \mathbb{K} $, we can write  
$$
u = \sum_{k \in J} u(k) e_k \quad \text{and} \quad v = \sum_{k \in J} v(k) e_k,
$$
for some finite $ J \subset I $, and thus,  
$$
T(u + \lambda v) = \sum_{k \in J} (u(k) + \lambda v(k)) e_k = Tu + \lambda Tv.
$$
Furthermore, $ T $ is surjective. Indeed, given $ w \in E $, by the definition of a basis, there exists a finite set of indices $ J \subset I $ and scalars $ \alpha_k $ such that  
$$
w = \sum_{k \in J} \alpha_k e_k.
$$
Defining  
$$
u = \sum_{k \in J} \alpha_k \chi_{\{k\}},
$$
we have $ u \in C_c(I) $ and $ Tu = w $. We also observe that 
$$
\|u\|_{\ell^p} = \|Tu\|_p,\quad \text{for all } u \in C_c(I),
$$ 
since  
$$
\|u\|_{\ell^p}^p 
= \int_{I} \left| \sum_{k \in J} u(k) \chi_{\{k\}} \right|^p \, \mathrm{d}\mu 
= \int_{I} \sum_{k \in J} |u(k)|^p \chi_{\{k\}} \, \mathrm{d}\mu 
= \sum_{k \in J} |u(k)|^p = \|Tu\|_p^p.
$$
In particular, $ T $ is injective. Now, since $ T $ is a bijective linear operator and $ (E, \|\cdot\|_p) $ is a Banach space, by Lemma \ref{lemma3.1}-\ref{2.1c}, we should have that $ (C_c(I), \|\cdot\|_{\ell^p}) $ is Banach, where $ \|\cdot\|_{\ell^p} = \|T(\cdot)\|_p $. However, this contradicts Lemma \ref{lemma4.1}-\ref{4.1a}, which states that $ C_c(I) $ is not Banach space when $ I $ is infinite. Therefore, we conclude that $ E $ must have finite dimension.	
\end{proof}

\begin{corollary}
	Let $ E $ be a nontrivial vector space. Then, the space $ (E, \|\cdot\|_p) $ is separable if and only if the Hamel bases of $ E $ are countable.  
\end{corollary}
\begin{proof}  
Let $ \mathcal{B} = \{e_k\}_{k \in I} $ be a Hamel basis for $ E $, where $I$ is the index set. From the previous proof, we know that there exists an isometric isomorphism $ T: (C_c(I), \|\cdot\|_{\ell^p}) \to (E, \|\cdot\|_p) $. Therefore, $ E $ is separable if and only if $ C_c(I) $ is separable. But by Lemma \ref{lemma4.1}-\ref{4.1b}, this holds if and only if $ I $ is countable, or equivalently, the basis $ \mathcal{B} $ is countable. Finally, since all Hamel bases have the same cardinality (see \cite[Theorem 2.42.]{Knapp}), the result follows.  
\end{proof}

The case $ p = \infty $ can be proved using the same ideas as before. However, we will instead take a more interesting approach. Below, we will prove a theorem stating that if the norm on a vector space behaves in a certain way with respect to the coordinates of a Hamel basis, then the space can only be Banach in finite dimensions.

\begin{theorem}\label{theorem4.3}
	Let $(E, \|\cdot\|)$ be a normed space, and let $\mathcal{B} = \{e_k\}_{k \in I}$ be a Hamel basis for $E$, where $I$ is its index set. Suppose that for each $k \in I$, there exists a constant $M_k > 0$ such that  
	\begin{equation}\label{5}  
		|u_k| \leq M_k \|u\| \quad \forall u \in E,  
	\end{equation}  
	where $u$ has the representation  
	$$  
	u = \sum_{i \in J} u_i e_i + u_k e_k, \quad \text{with } J \subset I \text{ finite and } k \notin J.  
	$$  
	Then, $(E, \|\cdot\|)$ is a Banach space if and only if $\dim(E) < \infty$.
\end{theorem} 

\begin{remark}\label{remark3}
	Hypothesis \eqref{5} is equivalent to requiring that the coordinate functions are continuous with respect to the norm $\|\cdot\|$. Explicitly, for each $k \in I$, consider the coordinate projection $\pi_k: E \to \mathbb{K}$ given by $\pi_k(u) = u_k$. Then, $\pi_k$ is continuous if and only if \eqref{5} holds.
\end{remark}

\begin{proof}[{\bf Proof of Theorem \ref{theorem4.3}}]
It suffices to show that if $(E,\|\cdot\|)$ is a Banach space, then $E$ must be finite-dimensional, since the converse follows from the fact that all finite-dimensional spaces are Banach. Assume $E$ is infinite-dimensional. Since $E$ is Banach, any Hamel basis $\mathcal{B}$ must be uncountable. Select a countable subset $\mathcal{B}' = \{e_k\}_{k\in\mathbb{N}} \subset \mathcal{B}$ and define  
$$
u = \sum_{k=1}^{\infty} \frac{1}{2^k} \frac{e_k}{\|e_k\|}.  
$$  
This series is absolutely convergent in the Banach space $E$, hence convergent, and so, $u\in E$. Now, let 
$$
u_n = \sum_{k=1}^n \frac{1}{2^k} \frac{e_k}{\|e_k\|}
$$
denote its partial sums, which satisfy
$$
u_n \to u\quad\text{in}\quad E.
$$ 
By hypothesis \eqref{5} and Remark \ref{remark3}, each coordinate projection $\pi_j: E \to \mathbb{K}$ is continuous. Therefore, for every $j \in \mathbb{N}$, one has  
$$
\pi_j(u) = \lim_{n\to \infty} \pi_j(u_n) = \lim_{n\to \infty} \sum_{k=1}^n \frac{1}{2^k \|e_k\|} \pi_j(e_k) = \frac{1}{2^j \|e_j\|}.
$$  
This shows that $u$ has infinitely many non-zero coordinates with respect to $\mathcal{B}$, contradicting the definition of a Hamel basis (where every vector has a finite linear combination). Thus, $E$ must be finite-dimensional.	
\end{proof}

By Theorem \ref{theorem4.3}, it immediately follows that the coordinate functionals on a Banach space are continuous if and only if the space has finite dimension.

\begin{proof}[{\bf Proof of Theorem \ref{theorem1.1} for the case $p=\infty$.}]
Let $ E $ be a nontrivial vector space, and let $ \mathcal{B} = \{e_k\}_{k \in I} $ be a Hamel basis for $ E $ with $I$ being $I$ the index set. By definition of norm $\| \cdot \|_{\infty} $ in \eqref{infty-norm}, for each $ k \in I $, we have  
$$
|u_k| \leq \|u\|_{\infty}, \quad \forall\,\, u \in E.  
$$  
Applying Theorem \ref{theorem4.3}, we conclude that $ (E, \|\cdot\|_\infty) $ is a Banach space if and only if $ \dim(E) < \infty $, which concludes the proof.	
\end{proof}

\begin{example}  
The vector space of polynomials $ \mathcal{P}(\mathbb{R}) $ equipped with the norm defined by  
$$
\|p\| = \max_{j} |\alpha_j|, \quad \text{where} \quad p(t) = \alpha_0 + \alpha_1 t + \dots + \alpha_n t^n,  
$$  
is not a Banach space. Indeed, consider the Hamel basis $ \{1, t, t^2, \dots\} $ for $ \mathcal{P}(\mathbb{R}) $. Clearly, we have $ \|p\| = \|p\|_{\infty} $. Since $ \mathcal{P}(\mathbb{R}) $ is infinite-dimensional, it follows by the contrapositive of Theorem \ref{theorem1.1} that $ (E, \|\cdot\|) $ cannot be a Banach space.  
\end{example}

\vspace{0.2cm}

\noindent {\bf Declarations:} Not applicable.



\begin{thebibliography}{1} 	
	
\bibitem{Aliprantis} Aliprantis, Charalambos D.; Border, Kim C.: Infinite Dimensional Analysis: a Hitchhiker's Guide, Springer Science \& Business Media, 2006. 	

\bibitem{Arent} Arendt, Wolfgang; Nittka, Robin:  Equivalent complete norms and positivity, Archiv der Mathematik {\bf 92} (2009), no. 5, 414-427.

\bibitem{Axler} Axler, Sheldon: Measure, Integration \& Real Analysis, Springer Nature, 2020.

\bibitem{Bartoszynski} Bartoszynski, Tomek; Džamonja, Mirna; Halbeisen, Lorenz; Murtinová, Eva; Plichko, Anatolij: On bases in Banach spaces, Studia Mathematica, {\bf 170} (2005), no. 2, 147-171.

\bibitem{Brezis} Brezis, Haim: Functional analysis, Sobolev spaces and partial differential equations, New York: Springer, 2011.

\bibitem{Carothers} Carothers, Neal L.: A short course on Banach space theory, Cambridge University Press, 2005.

\bibitem{Lorenz} Halbeisen, Lorenz; Hungerbühler, Norbert: The cardinality of Hamel bases of Banach spaces, East West Math, {\bf 2} (2000), no. 2, 153-159.

\bibitem{Jones} Jones, Frank: Lebesgue Integration on Euclidean Space, Jones \& Bartlett Learning, 2001.

\bibitem{Jordan} Jordan, Pascual; Neumann, J. von: On inner products in linear, metric spaces, Annals of Mathematics, {\bf 36} (1935), no. 3, 719-723.

\bibitem{Kikianty2} Kikianty, Eder: Towards a geometrical equivalence of norms, Quaestiones Mathematicae, {\bf 47} (2024), 247-264.

\bibitem{Kikianty} Kikianty, Eder; Sinnamon, Gord: Angular equivalence of normed spaces, Journal of Mathematical Analysis and Applications, {\bf 454} (2017), no. 2, 942-960.

\bibitem{Knapp} Knapp, Anthony W.: Basic algebra, Springer Science \& Business Media, 2006.	
	
\bibitem{Kreyzig} Kreyszig, Erwin: Introductory Functional Analysis with Applications, John Wiley \& Sons, 1991.	

\bibitem{Laugwitz} Laugwitz, Detlef: Über vollständige Normtopologien in linearen Räumen, Archiv der Mathematik, {\bf 6} (1955), no. 2, 128–131.

\end{thebibliography}
\end{document}